\documentclass{amsart}

\usepackage{amssymb}
\usepackage{amsmath}
\usepackage{tikz-cd}

\newtheorem{theorem}{Theorem}[section]

\newtheorem{definition}[theorem]{Definition}
\newtheorem{example}[theorem]{Example}

\newtheorem{lemma}[theorem]{Lemma}

\newtheorem{proposition}[theorem]{Proposition}

\begin{document}

\title[A classifying localic category for locally compact locales]{A classifying localic category for locally compact locales}
\author{Dr Christopher Townsend}
\maketitle
\begin{abstract}
For an internal category $\mathbb{C}$ in a cartesian category $\mathcal{C}$ we define, naturally in objects $X$ of $\mathcal{C}$, $Prin_{\mathbb{C}}(X)$. This is a category whose objects are principal $c \mathbb{C}$-bundles over $X$ and whose morphisms are principal $c(\mathbb{C}^{\uparrow})$-bundles. Here $c(\_)$ denotes taking the core groupoid of a category (same objects but only isomorphisms as morphisms) and $\mathbb{C}^{\uparrow}$ is the arrow category of $\mathbb{C}$ (objects are morphisms, morphisms are commuting squares). We show that $X \mapsto Prin_{\mathbb{C}}(X)$ is a stack of categories and call stacks of this sort lax-geometric. We then provide two sufficient conditions for a stack to be lax-geometric and use them to prove that the pseudo-functor $X \mapsto \mathbf{LK}_{Sh(X)}$ on the category of locales $\mathbf{Loc}$ is a lax-geometric stack. Here $\mathbf{LK}_{Sh(X)}$ is the category of locally compact locales in the topos of sheaves over $X$, $Sh(X)$. Therefore there exists a localic category $\mathbb{C}_{\mathfrak{LK}}$ such that $\mathbf{LK}_{Sh(X)} \simeq Prin_{\mathbb{C}_{\mathfrak{LK}}}(X)$ naturally for every locale $X$. 


{\it 2020 Mathematics Subject Classification (MSC2020)}: 06D22, 18F70, 18D40, 14A20.
\end{abstract}

\section{Introduction}

Some recent work (\cite{HenryTow} and \cite{HenryTowClass}) has shown that there is a localic groupoid that classifies compact Hausdorff locales. That is, there is a localic groupoid $\mathbb{G}_{\mathfrak{KH}}$ such that for any locale $X$, the category of compact Hausdorff locales over $Sh(X)$ is equivalent 
to a category of principal $\mathbb{G}_{\mathfrak{KH}}$-bundles over $X$. This result was offered as a compact Hausdorff dual to the existence of a classifying localic groupoid for discrete locales, something that is effectively well known by applying \cite{JoyT} to the object classifier and recalling \cite{BDesc} (which shows that the topos of $\mathbb{G}$-equivariant sheaves, $B(\mathbb{G})$, classifies principal $\mathbb{G}$-bundles). 

Now recall Blass' theorem (\cite{blass}) which states that an elementary topos $\mathcal{S}$ has a natural numbers object (i.e. satisfies the Axiom of Infinity) if and only if it has an object classifier. Blass' theorem can be restated for our context: there is a natural numbers object in $\mathcal{S}$ iff $\mathbf{Loc}_{\mathcal{S}}$, i.e. locales in $\mathcal{S}$, has a classifying groupoid for discrete objects. A question to consider then is whether the existence of $\mathbb{G}_{\mathfrak{KH}}$ is somehow a compact Hausdorff dual to the ordinary Axiom of Infinity? Perhaps we should try to understand it in the context of axiomatic approaches to locale theory (e.g. \cite{TowIdemp})? This may or may not lead to anything, but clearly it would be nicer to have a single axiom rather than two order dual ones. 

In this paper we meet that challenge and prove as our main result (Theorem \ref{main}) that there exists a classifying localic category for the stack of locally compact locales. The classes of compact Hausdorff and discrete locales are both locally compact, and it is easy to see from the main result that they too have classifying localic categories.

To prove our main result the key technical hurdle we need to overcome is that morphisms between principal bundles are always isomorphisms. So if we are hoping to classify stacks of categories as principal bundles, we are going to have change the usual definition of morphism. For both the case of discrete locales and the case of compact Hausdorff locales this hurdle can be overcome by replacing morphism with $\mathbb{S}$-homotopy, where $\mathbb{S}$ is the Sierpi\'{n}ski locale (see \cite{HenryTowClass}). But the order on the Sierpi\'{n}ski locale is reversed between the two cases, so we cannot expect this approach to extend to locally compact locales (and indeed Corollary 4.3 of \cite{TowHarv} rules this out). Section \ref{prin_c} below proposes a way to overcome this hurdle for locally compact locales, using a technique loosely related to one that has been explored in algebraic geometry; see \emph{c-stacks} in \cite{Drin}. This change in definition is available because we choose to generalise the underlying data from a localic groupoid, $\mathbb{G}$, to a localic category, $\mathbb{C}$. The objects are still principal bundles, but now of the core of $\mathbb{C}$, and the morphisms are certain natural transformations which are, up to isomorphism, principal bundles on the core of the arrow category of $\mathbb{C}$ (see also \cite{MoerGpdII} where, similarly, generality at the level of objects is ignored, but used for morphisms). The analysis works relative to any cartesian category and we offer a simple characterization of what we are calling \emph{lax-geometric} stacks; this is the first key proposition below (Proposition \ref{lax-conditions}). 

The remainder of the work needed for the main result is very similar to \cite{HenryTowClass}, but now with locally compact locales rather than compact Hausdorff locales. We first check that locally compact locales descend and so form a stack; this appears to be a new result for locale theory, but is quite straightforward. The next step is to use an information system type representation for continuous frames (i.e. for the frames of opens of locally compact locales). These are geometric theories, so stable under inverse images, and every continuous frame arises as the rounded ideals of an information system. Putting these observations together, the main result is then really just a repetition of \cite{HenryTowClass} (which itself is effectively a repetition of known techniques).

\section{Technical background}

\subsection{Principal bundles}
We will be investigating principal bundles internal to a cartesian (i.e. finitely complete) category $\mathcal{C}$. If $\mathbb{G} = (d_0,d_1:G_1 \rightrightarrows G_0,m,s,i)$ is a groupoid internal to $\mathcal{C}$ then a $\mathbb{G}$-object $P = (x:P \rightarrow G_0,\alpha:G_1 \times_{G_0} P \rightarrow P)$ is a  principal $\mathbb{G}$-bundle over $X$ (say via $p:P \rightarrow X$) if $(\alpha,\pi_2): G_1 \times_{G_0} P \rightarrow P \times_X P $ is an isomorphism and $p$ is an effective descent morphism (i.e. pullback $p^*: \mathcal{C}/X \rightarrow \mathcal{C}/P$ is monadic). In other words, the usual definition, but with effective descent in place of open surjection. It can be verified that even in this general context principal $\mathbb{G}$-bundles define cocycles and vice versa; we write $\psi^P:\mathbb{P} \rightarrow \mathbb{G}$ for a cocycle corresponding to $P$. This $\psi^P$ is an internal functor, so its data is a pair $(\psi^P,x_P:P \rightarrow G_0)$; i.e. we will tend not to notate its action on objects. Also $\mathbb{P}$ is the groupoid $\pi_1,\pi_2: P \times_X P \rightrightarrows P$ which we may write $\mathbb{P}_p$ if we need to emphasise which map back to $X$ is relevant (our general convention is to use $X_f$ as notation for a $f:X \rightarrow Y$ when viewed as an object of $\mathcal{C}/Y$). This well known correspondence between cocycles and principal bundles extends to morphisms: for any two principal bundles $P_0$, $P_1$, there is a natural bijection between principal bundle morphisms $P_0 \rightarrow P_1$ and natural transformations  $\psi^{P_0}_{\pi_1} \rightarrow \psi^{P_1}_{\pi_2}$ (where $\pi_i: \mathbb{P}_0 \times_{\mathbb{X}} \mathbb{P}_1 \rightarrow \mathbb{P}_{i-1}$, $i=1,2$; these are internal functors). Natural transformations between functors to groupoids are necessarily isomorphisms so this bijection provides a way of establishing the well known fact that all morphisms between principal bundles are isomorphisms. 

The correspondence between cocycles and principal bundles can be made functorial in the sense that for any internal functor $\mathbb{F}:\mathbb{H} \rightarrow \mathbb{G}$ and any principal $\mathbb{H}$-bundle $P$ over $X$, we can define a principal $\mathbb{G}$-bundle over $X$ as the principal bundle corresponding to the cocycle $\mathbb{F}\psi^P$. We write this principal bundle as $\Sigma_{\mathbb{F}}(P)$.  

Note also that by pullback stability properties of effective descent morphisms, any internal functor $\mathbb{P}_0 \times_{\mathbb{X}} \mathbb{P}_1 \rightarrow \mathbb{G}$ is necessarily a cocycle (i.e. its domain is of the form $\mathbb{P}_p$ for some effective descent morphism $p:P \rightarrow X$).

For any internal category $\mathbb{C} = (C_1 \rightrightarrows C_0, ...)$ of $\mathcal{C}$, a cocycle on its core (written $c\mathbb{C}$) is equally an internal functor $\mathbb{P} \rightarrow \mathbb{C}$ because any such functor must factor through the core of $\mathbb{C}$. If we are looking at a cocycle to a core of an internal category we will always assume that the codomain is $\mathbb{C}$; i.e. principal bundles $P$ for the core $c\mathbb{C}$ give rise to cocycles $\mathbb{P} \rightarrow \mathbb{C}$ and vice versa.

\subsection{Stacks}\label{stacks}
A \emph{stack} on $\mathcal{C}$ is a pseudo-functor $M: \mathcal{C}^{op} \rightarrow \mathfrak{CAT}$ such that for any effective descent morphism $f:Y \rightarrow X$, $M(X)$ is equivalent, via the canonical functor, to $Des(M,f)$. Here $Des(M,f)$ is the category whose objects are pairs $(A, \theta^A)$ where $A$ is an object of $M(Y)$ and $M(\pi_1)(A) \xrightarrow{\theta^A} M(\pi_2)(A)$ is descent data for $A$; the morphisms of $Des(M,f)$ are morphisms of $M(Y)$ commuting with the $\theta$s in the obvious manner. This is the usual definition of stack when covers consist of a single morphism that is an effective descent morphism; see Definition B1.5.1 of \cite{Elephant} for background and the preamble to Proposition B1.5.5 for the case where covers consist of only a single morphism. 

\begin{example}
The pseudo-functor $X \mapsto \mathcal{C}/X$ (with $f: X \rightarrow Y$ mapped to the pullback functor $f^*: \mathcal{C}/Y \rightarrow \mathcal{C}/X$) is a stack. This is just by definition of effective descent. We write this stack $\mathfrak{C}:\mathcal{C}^{op} \rightarrow \mathfrak{CAT}$.
\end{example}

\begin{example}
For any groupoid $\mathbb{G}$ internal to $\mathcal{C}$, the pseudo-functor $X \mapsto Prin_{\mathbb{G}}(X)$ is a stack. 
\end{example}

\begin{definition}
A pseudo-functor $M: \mathcal{C}^{op} \rightarrow \mathfrak{CAT}$ is a \emph{geometric} stack if $M(X) \simeq Prin_{\mathbb{G}}(X)$ naturally in objects $X$.
\end{definition}

\begin{proposition}\label{stack_gpd}
    A stack $M: \mathcal{C}^{op} \rightarrow \mathfrak{CAT}$ of groupoids is geometric if and only if the following two conditions hold: 

(i) There exists an object $C$ in $M(G_0)$ for some object $G_0$ of $\mathcal{C}$, such that for any other object $A$ of $M(X)$ there exists a cover $p_A:P_A \rightarrow X$, a morphism $x_A:P_A \rightarrow G_0$ and an isomorphism $M(p_A)(A) \cong M(x_A)(G_0)$.

(ii) There exists an object $G_1$ and a bijection
\begin{eqnarray*}
\mathcal{C}(X,G_1) \cong \{(f,\theta,g) | f,g:X \rightarrow G_0, \theta:M(f)(C) \rightarrow M(g)(C) \} 
\end{eqnarray*}
natural in $X$.
\end{proposition}
\begin{proof}
This is just a restating of a well known characterisation of geometric stacks (where (ii) is more usually stated as saying that the diagonal of $M$ is representable). For example, see Proposition 70 of \cite{Metzler}. A proof can also effectively be taken from Proposition 4.2 of \cite{HenryTowClass}. The groupoid structure comes from naturality of (ii); also by naturality of (ii) we know that the mate of any $k:X \rightarrow G_1$ must be of the form $(d_0k,\theta,d_1k)$ where $d_0,d_1$ are the domain and codomain maps of the groupoid. 
\end{proof}

The final background result on stacks that we need is: 
\begin{proposition}
If $M: \mathcal{C}^{op} \rightarrow \mathfrak{CAT}$ is a stack then the pseudo-functor $M^{\uparrow}:\mathcal{C}^{op} \rightarrow \mathfrak{CAT}$ defined by $M^{\uparrow}(X)= [\mathbf{2},M(X)]$ is also a stack.  
\end{proposition}
\begin{proof}
    This is easy to check `by hand' and well known, e.g. Corollary 2.9 of \cite{Stacks79} (take $D = \mathbf{2}$ therein).
\end{proof} 

\section{Morphisms between Principal Bundles}\label{prin_c}
A problem with principal bundles is that all morphisms between them are isomorphisms. So if we are hoping to use principal bundles to represent stacks of categories rather just stacks of groupoids we are stuck. In this paper we are proposing a way round the problem by defining a category of principal bundles relative to any internal category. We then ignore the generality (from groupoids to categories) at the level of objects but use it to define morphisms. The idea is encapsulated in the following definition:

\begin{definition}
For an internal category $\mathbb{C}$ in a cartesian category $\mathcal{C}$, define for each object $X$ of $\mathcal{C}$, the data $ Prin_{\mathbb{C}}(X)$ by: 

{\bf Objects}. The objects are principal $c\mathbb{C}$-bundles over $X$.

{\bf Morphisms}. If $P_0$ and $P_1$ are two principal $c\mathbb{C}$-bundles over $X$, then a morphism from $P_0$ to $P_1$ is an internal natural transformation $\psi^{P_0}_{\pi_1} \rightarrow \psi^{P_1}_{\pi_2}$, where $\pi_i: \mathbb{P}_0 \times_{\mathbb{X}} \mathbb{P}_1 \rightarrow \mathbb{P}_{i-1}$, $i=1,2$. (Recall our convention that $\psi^P:\mathbb{P} \rightarrow \mathbb{C}$, so these morphisms are not all isomorphisms.)
\end{definition}

I must thank David Roberts for providing me with this definition of morphism\footnote{And the reader should thank David too, as without this insight proofs about $Prin_{\mathbb{C}}(X)$ become much longer!} which also appears as Definition 5.4 of \cite{David12}. Defining morphisms as certain natural transformations makes it easy to define composition in $Prin_{\mathbb{C}}(X)$: the identity on $P$ is $\psi^P$, and we can compose $\psi^{P_0}_{\pi_1} \xrightarrow{\alpha_0} \psi^{P_1}_{\pi_2}$ followed by $\psi^{P_1}_{\pi_1} \xrightarrow{\alpha_1} \psi^{P_2}_{\pi_2}$ by pulling back to $P_0 \times_X P_1 \times_X P_2$ to obtain a natural transformation $\gamma: \psi^{P_0}_{\pi_1} \rightarrow \psi^{P_1}_{\pi_3}$ by composing $\alpha_0$ with $\alpha_1$ (so $\pi_i:P_0 \times_X P_1 \times_X P_2 \rightarrow P_{i-1}$ and $\gamma (p_0,p_1,p_2) = \alpha_1(p_1,p_2)\alpha_0(p_0,p_1)$). But by naturality $\gamma (p_0,p_1,p_2) = \gamma (p_0,p'_1,p_2)$ for any pair $p_1,p'_1$ of $P_1$, so that $\gamma$ must factor through a morphism $P_0 \times_X P_2 \rightarrow C_1$ which defines the composition of $\alpha_0$ and $\alpha_1$ in $Prin_{\mathbb{C}}(X)$. (Note that we are using `sets with elements' notation, but this is just shorthand for the relevant diagrams in $\mathcal{C}$.) So $Prin_{\mathbb{C}}(X)$ is a category; naturality in $X$ of all these constructions is straightforward.  

However, we must of course make good the claim in the abstract and introduction that the morphisms of $Prin_{\mathbb{C}}(X)$ are principal bundles over the core of the arrow category: 
\begin{proposition}\label{cocycle}
    For any pair of principal $c\mathbb{C}$-bundles $P_0$ and $P_1$, morphisms $P_0 \rightarrow P_1$ of $Prin_{\mathbb{C}}(X)$ are, up to isomorphism, triples $(\beta_0,P,\beta_1)$ where $P$ is a principal $c\mathbb{C}^{\uparrow}$-bundle and $\beta_0:P_0 \rightarrow  \Sigma_{D_0}P$ and $\beta_1 :\Sigma_{D_1}P \rightarrow P_1$ are principal bundle morphisms (isomorphisms). (Here $D_0,D_1:c\mathbb{C}^{\uparrow} \rightrightarrows c\mathbb{C}$ are the internal functors determined by the domain and codomain morphisms of $\mathbb{C}$; we use $c\mathbb{C}^{\uparrow}$ for the core of the arrow category of $\mathbb{C}$). 
\end{proposition}
\begin{proof}
    Given a natural transformation $\alpha: \psi^{P_0}_{\pi_1} \rightarrow \psi^{P_1}_{\pi_2}$ define a cocycle $(\psi^{\alpha}, x_{\alpha}):\mathbb{P}_0 \times_{\mathbb{X}} \mathbb{P}_1 \rightarrow \mathbb{C}^{\uparrow}$ by $x_{\alpha}(p_0,p_1) = \alpha(p_0,p_1)$ and with $\psi^{\alpha}$ given by commuting 
\[
\begin{tikzcd}[column sep=5em, row sep=3em]
\cdot \arrow[r, "\psi^{P_0}(p_0{,}p'_0)" above, "\cong" below] \arrow[d, "\alpha(p_0{,}p_1)"'] 
  & \cdot \arrow[d, "\alpha(p'_0{,}p'_1)"] \\
\cdot \arrow[r, "\psi^{P_1}(p_1{,}p'_1)"' below, "\cong" above] 
  & \cdot
\end{tikzcd}
\]
    for any pair of pairs $(p_0,p_1),(p_0',p_1')$ in $P_0 \times_X P_1$. This cocycle determines a principal bundle $P^{\alpha}$ for $c\mathbb{C}^{\uparrow}$. Because $D_0\psi^{\alpha}$ factors through $\psi^{P_0}$ via $\pi_1:\mathbb{P}_0 \times_{\mathbb{X}} \mathbb{P}_1 \rightarrow \mathbb{P}_0$, the cocycles $D_0\psi^{\alpha}$ and $\psi^{P_0}$ both determine the same principal bundle up to isomorphism; this is because we know that $\pi_1:P_0 \times_X P_1 \rightarrow P_0$ is a pullback stable regular epimorphism as $P_1 \rightarrow X$ is of effective descent. Therefore there is a natural isomorphism $\beta_0^{\alpha}:P_0 \xrightarrow{\cong}\Sigma_{D_0}P^{\alpha}$ and similarly there exists $\beta_1^{\alpha}:\Sigma_{D_1}P^{\alpha} \xrightarrow{\cong} P_1$.

    In the other direction say we are given $P$ a principal bundle for the core of the arrow category, equipped with $\beta_0$ and $\beta_1$ as in the statement of the proposition. As bundle morphisms correspond to natural transformations, the data for $\beta_0$ is a map $P_0 \times_X P \rightarrow C_1$ and similarly for $\beta_1$. So we can define a natural transformation $\gamma: P_0 \times_X P \times_X P_1 \rightarrow C_1$ by $\gamma(p_0,p,p_1) = \beta_1(p,p_1)x_P(p)\beta_0(p_0,p)$; but, similarly to how we just checked that composition in $Prin_{\mathbb{C}}(X)$ is well defined, we have $\gamma (p_0,p,p_1) = \gamma (p_0,p',p_1)$ for any pair $p,p'$ of $P$. So the cocycle corresponding to $P$ factors through $P_0 \times_X P_1$ and so defines a morphism of $Prin_{\mathbb{C}}(X)$. It is routine to verify that these two constructions are inverse to each other up to isomorphism. 
\end{proof}

Using the correspondence between maps of principal bundles and natural transformations between cocycles it is then clear that: 
\begin{proposition}
Given an internal category $\mathbb{C}$ in a cartesian category $\mathcal{C}$,
\begin{eqnarray*}
    Prin_{c \mathbb{C}}(X) \simeq c(Prin_{\mathbb{C}}(X))
\end{eqnarray*}
naturally in objects $X$ of $\mathcal{C}$.
\end{proposition}
In other words we haven't introduced or lost any of the usual morphisms (isomorphisms) between principal bundles. 

\section{$Prin_{\mathbb{C}}( \_ )$ is a stack}

\begin{proposition}\label{as_stack}
For any internal category $\mathbb{C}$ of a cartesian category $\mathcal{C}$ the pseudo-functor
\begin{eqnarray*}
    \mathcal{C}^{op} & \rightarrow & \mathfrak{CAT} \\
    X & \mapsto & Prin_{\mathbb{C}}(X) \\
\end{eqnarray*}
is a stack. 
\end{proposition}
\begin{proof}
As recalled in \S \ref{stacks} $X \mapsto Prin_{c\mathbb{C}}(X)$ is a stack; it is a stack of groupoids and we have just shown $Prin_{c\mathbb{C}}(X) \simeq cPrin_{\mathbb{C}}(X)$. So given an effective descent morphism $f: Y \rightarrow X$ of $\mathcal{C}$ we know that the canonical map
\begin{eqnarray*}
f^*: cPrin_{\mathbb{C}}(X) \rightarrow Des( cPrin_{\mathbb{C}}(Y), f)
\end{eqnarray*}
is an equivalence. Because descent data is always an isomorphism we also know that $Des( cPrin_{\mathbb{C}}(Y), f) \simeq cDes( Prin_{\mathbb{C}}(Y), f)$. So we know that the core of the canonical map $f^*:Prin_{\mathbb{C}}(X) \rightarrow Des( Prin_{\mathbb{C}}(Y), f)$ is an equivalence from which it is clear that $f^*$ is essentially surjective.
To prove that $f^*$ is full and faithful we must show for any two principal bundles $P_0$ and $P_1$ that $f^*$ induces a bijection of homsets
\begin{eqnarray*}
    Prin_{\mathbb{C}}(X)(P_0,P_1) & \cong & Des(Prin_{\mathbb{C}}(Y),f)(f^*P_0,f^*P_1)
\end{eqnarray*}
But $X \mapsto Prin_{c\mathbb{C}^{\uparrow}}(X)$ is a stack and so $Prin_{c\mathbb{C}^{\uparrow}}(X) \simeq Des( Prin_{c\mathbb{C}^{\uparrow}}(Y), f)$ which provides the bijection to cocycles of $c\mathbb{C}^{\uparrow}$. This completes the proof as we have just established in Prop. \ref{cocycle} that cocycles of $c\mathbb{C}^{\uparrow}$ are, up to isomorphism, the same thing as natural transformations $\psi^{P_0}_{\pi_1} \rightarrow \psi^{P_1}_{\pi_2}$.
\end{proof}

\section{Lax-geometric stacks}
\begin{definition}
    A pseudo-functor $M: \mathcal{C}^{op} \rightarrow \mathfrak{CAT}$ is a \emph{lax-geometric stack} if it is naturally equivalent to $Prin_{\mathbb{C}}(\_)$ for some category $\mathbb{C}$ internal to $\mathcal{C}$.
\end{definition}
Note that by Proposition \ref{as_stack} any lax-geometric stack is a stack, so the choice of wording in the definition is appropriate. We use $\mathbb{C}_M$ for the internal category of a lax-geometric stack $M$ and may also talk of $\mathbb{C}_M$ `representing' the stack. But it is important to understand that we are not saying that this `representing' category  is unique (even up to Morita equivalence, as it is not a groupoid). We hope to understand the correct notion of Morita equivalence for internal categories in later work.

\begin{proposition}\label{lax-conditions}
    A pseudo-functor $M: \mathcal{C}^{op} \rightarrow \mathfrak{CAT}$ is a \emph{lax-geometric stack} if and only if it is a stack and; 

(i) There exists an object $C$ in $M(C_0)$ for some object $C_0$ of $\mathcal{C}$ such that for any other object $A$ of $M(X)$, there exists a cover $p_A:Y \rightarrow X$, a morphism $q_A:Y \rightarrow C_0$ and an isomorphism $M(p_A)(A) \cong M(q_A)(C_0)$.

(ii) There exists an object $C_1$ and a bijection
\begin{eqnarray*}
\mathcal{C}(X,C_1) \cong \{(f,\theta,g) | f,g:X \rightarrow C_0, \theta:M(f)(C) \rightarrow M(g)(C) \} 
\end{eqnarray*}
natural in $X$.
\end{proposition}
To ease the proof we isolate some of the reasoning needed as a lemma:
\begin{lemma}\label{stability}
    If a pseudo-functor $M: \mathcal{C}^{op} \rightarrow \mathfrak{CAT}$ satisfies (i) and (ii) then so too does $M^{\uparrow}: \mathcal{C}^{op} \rightarrow \mathfrak{CAT}$. 
\end{lemma}
\begin{proof}
Firstly (i) and (ii) are sufficient to make $C_1$, $C_0$ into the object of morphisms and the object of objects respectively of an internal category $\mathbb{C}$. This is clear by following the same reasoning deployed in the proof of Proposition \ref{stack_gpd}. From that analysis we recall that by naturality of (ii) the mate of any $k:X \rightarrow C_1$ is always of the form $(d_0k,\theta^k,d_1k)$ with
\begin{eqnarray*}
\theta^k = M(d_0 k) \cong M(k)M(d_1) \xrightarrow{M(k)(Id^M_{C_1})} M(k)M(d_1) \cong M(d_1 k)
\end{eqnarray*}
where $Id^M_{C_1}$ is obtained via (ii) as the mate of the identity map $Id_{C_1}:C_1 \rightarrow C_1$. Also if $k,k':X \rightarrow C_1$ are such that $d_1k = d_0k'$ then the mate of $kk' = m(k,k')$ is
\begin{eqnarray*}
    M(d_0k)(C) \xrightarrow{\theta^k} M(d_1k)(C) = M(d_0k')(C) \xrightarrow{\theta^{k'}} M(d_1k')(C)\text{.}
\end{eqnarray*}
In other words $\theta^{k'}\theta^{k}=\theta^{k'k}$.

So let us assume (i) and (ii) for $M$. We initially focus on proving (i) for $M^{\uparrow}$. For this take $Id^M_{C_1}:M(d_0)(C) \rightarrow M(d_1)(C)$ which is an object of $M^{\uparrow}(C_1)$. Then for any morphism $v:A \rightarrow B$ of $M(X)$ we must find an effective descent morphism $q_v:Q \rightarrow X$ and a morphism $x_v:Q \rightarrow C_1$ together with an isomorphism: $M^{\uparrow}(q_v)(A \xrightarrow{v} B) \cong M^{\uparrow}(x_v)(Id^M_{C_1})$. By (i) applied to $M$ we can find effective descent morphisms $q_A:Y_A \rightarrow X$ and $q_B:Y_B \rightarrow X$, morphisms $x_A:Y_A \rightarrow C_0$ and $x_B:Y_B \rightarrow C_0$ and isomorphisms $M(q_A)(A) \cong M(x_A)(C)$ and $M(q_B)(B) \cong M(x_B)(C)$. Let $Q$ be the pullback of $q_A$ along $q_B$ and define $q_v$ to be the composition $q_A \pi_1$ ($=q_B\pi_2$). Define $f = Q \xrightarrow{\pi_1} Y_A \xrightarrow{x_A} C_0$, $g = Q \xrightarrow{\pi_2} Y_B \xrightarrow{x_B} C_0$. By applying (ii) to $M$ we can then find $x_v: Q \rightarrow C_1$ as the mate of
\begin{eqnarray*}
M(f)(C) \xrightarrow{\cong} M(q_v)(A) \xrightarrow{M(q_v)(v)} M(q_v)(B) \xrightarrow{\cong} M(g)(C)
\end{eqnarray*}
and we have $M^{\uparrow}(q_v)(A \xrightarrow{v} B) \cong M^{\uparrow}(x_v)(Id^M_{C_1})$ by construction. This verifies (i) for $M^{\uparrow}$.
 
For (ii), say $k: X \rightarrow (\mathbb{C}^{\uparrow})_1$; so $k=(k^t,k^b,k^l,k^r):X \rightarrow C_1 \times_{C_0} C_1 \times_{C_0} C_1 \times_{C_0} C_1$ where for any $x$ of $X$ we have 
\[
\begin{tikzcd}
\cdot \arrow[r, "k^t(x)"] \arrow[d, "k^l(x)"'] & \cdot \arrow[d, "k^r(x)"] \\
\cdot \arrow[r, "k^b(x)"'] & \cdot
\end{tikzcd}
\]
commuting. But then $(k^l,\theta^k,k^r)$ is its required mate, where
\begin{eqnarray*}
\theta^k:M^{\uparrow}(k^l)(Id^M_{C_1}) \rightarrow M^{\uparrow}(k^r)(Id^M_{C_1})
\end{eqnarray*}
is given by the pair $(\theta^{k^t}, \theta^{k^b})$. That $\theta^k$ is indeed a morphism of $M^{\uparrow}(X)$ follows from our earlier observation that for arbitrary composable $k,k'$, $\theta^{k'k}=\theta^{k'}\theta^k$. 
\end{proof}

\begin{proof} \emph{(Of Proposition \ref{lax-conditions})}

    Firstly say (i) and (ii) are satisfied. As in the proof of the lemma we can define a category $\mathbb{C}$ whose object of objects is $C_0$ and whose object of morphisms is $C_1$. Then by considering the core of $\mathbb{C}$ we clearly have
    \begin{eqnarray*}
        \mathcal{C}(X,cC_1) \cong \{(f,\theta,g) | f,g:X \rightarrow C_0, \theta:M(f)(C) \rightarrow M(g)(C) \text{, $\theta$ iso.} \} 
    \end{eqnarray*}
    naturally in $X$. We can combine this with (i) and apply Proposition \ref{stack_gpd} to conclude that $cM:\mathcal{C}^{op} \rightarrow \mathfrak{CAT}$ is geometric.
    So we know that the objects of $M(X)$ are principal $c \mathbb{C}$-bundles and to complete one direction of the proof we have to focus on the morphisms.

    However we have proved in the lemma that $M^{\uparrow}$ satisfies (i) and (ii) and we have seen that the category we get is $\mathbb{C}^{\uparrow}$. So by again restricting to the core (this time of $\mathbb{C}^{\uparrow}$) we see that the objects of $M^{\uparrow}(X)$ are, up to isomorphism, principal $c\mathbb{C}^{\uparrow}$-bundles over $X$. But the objects of $M^{\uparrow}(X)$ are morphisms of $M(X)$ and so $M(X) \simeq Prin_{\mathbb{C}}(X)$.

    In the other direction, say $M(X) \simeq Prin_{\mathbb{C}}(X)$. Then (i) follows by Proposition \ref{stack_gpd} as it is, effectively, only about objects. Condition (ii) requires a natural bijection between morphisms $X\rightarrow C_1$ and triples $(f,\theta,g)$ where $\theta$ is a natural transformation $\psi^{f^*c\mathbb{C}}_{\pi_1} \rightarrow \psi^{g^*c\mathbb{C}}_{\pi_2}$ and $f,g:X \rightrightarrows C_0$. Such a natural transformation is given by a morphism $f^*cC_1 \times_{X} g^*cC_1 \rightarrow C_1$; as the principal bundles $f^*c\mathbb{C}$ and $g^*c\mathbb{C}$ are trivial it is clear that any such natural transformation is uniquely determined by a map $X \rightarrow C_1$. 

\end{proof}

\section{Locally Compact Locales form a Stack}\label{lc_is_stack}

The purpose of this section is to prove that the pseudo-functor that assigns to each locale $X$ the category of locally compact locales over $X$ is a stack. We shall do this by first proving that exponentiability descends for any cartesian category $\mathcal{C}$ and then applying to the case $\mathcal{C} = \mathbf{Loc}$. 

Recall that for any morphism $f: X \rightarrow Y$ of a cartesian category $\mathcal{C}$, pullback $f^*: \mathcal{C}/Y \rightarrow \mathcal{C}/X$ preserves any exponentials that exist in $\mathcal{C}/Y$; i.e. $f^*(A^B) \cong f^*A^{f^*B}$. This is because the pullback adjunction $\Sigma_f \dashv f^*$ satisfies Frobenius reciprocity. The result we need to prove that locally compact locales form a stack goes in the other direction: exponentiability descends. In fact this is true quite generally: 
\begin{proposition}
    Let $\mathcal{C}$ be a cartesian category, $S$ an object of $\mathcal{C}$ and $f: X \rightarrow Y$ a morphism of $\mathcal{C}$ that is of effective descent. For any object $Z_g$ over $Y$, if the exponential $S_X^{f^*(Z_g)}$ exists in $\mathcal{C}/X$ then $S_Y^{Z_g}$ exists in $\mathcal{C}/Y$.
\end{proposition}
\begin{proof}
    $f^*Z_g$ comes equipped with descent data $\theta: \pi_1^*f^*Z_g \xrightarrow{\cong} \pi_2^*f^*Z_g$. So if $S_X^{f^*(Z_g)}$ exists, it can also be equipped with descent data: $S_{X\times_Y X}^{\theta^{-1}}$. So, $S_X^{f^*(Z_g)}$ must be of the form $f^*E_e$ for some object $E_e$ over $Y$. Proving that $E_e$ must be $S_Y^{Z_g}$, as we would expect, requires the observation that exponentials in $Des(\mathfrak{C},f)$ are created in $\mathcal{C}/X$; this is routine. 
\end{proof}

To proceed we now need to focus on the category $\mathbf{Loc}$ of locales; it is cartesian and will play the role of $\mathcal{C}$ in the above analysis. The stack $\mathfrak{LOC}:\mathbf{Loc}^{op} \rightarrow \mathfrak{CAT}$ given by $X \mapsto \mathbf{Loc}/X$ is equivalent to $X \mapsto \mathbf{Loc}_{Sh(X)}$ (\cite{JoyT}; or Parts C1 and C2 \cite{Elephant}). Under this isomorphism the Sierpi\'{n}ski locale relative to $Sh(X)$ corresponds to $\mathbb{S}_X$ (i.e. the projection $\mathbb{S} \times X \rightarrow X$). A locale $X$ is \emph{locally compact} if its frame of opens is continuous (Ch. VII \cite{Stone}). By Hyland's result (e.g. Ch. VII, Th. 4.11 of \cite{Stone}) we know that a locale $Z$ is locally compact iff the exponential $\mathbb{S}^Z$ exists. Hyland's result is constructive so it holds relative to any $Sh(X)$. Therefore because pullback preserves exponentials we know that $\mathfrak{LOC}$ restricts to a pseudo-functor $\mathfrak{LK}$, where $\mathfrak{LK}(X)$ is the full subcategory of locally compact locales in $Sh(X)$. But we have also just shown that exponentials descend and have therefore proved:     
\begin{proposition}
The pseudo-functor $\mathfrak{LK}:\mathbf{Loc}^{op} \rightarrow \mathfrak{CAT}$ defined by $X \mapsto \mathbf{LK}_{Sh(X)}$ is a stack.
\end{proposition}

\section{Constructing any locally compact locale from a distributive information system}
In this section we introduce distributive information systems. We then firstly prove that every continuous frame is the poset of rounded ideals of some distributive information system. Distributive information systems are the models of a geometric theory and in the second part of this section we show that  the formation of rounded ideals is stable under pullback (i.e. along the inverse image of any geometric morphism). Both of these facts are needed in the proof of our main theorem.  
\subsection{Distributive information systems}
We now use information system (\cite{vicinfosys}) type results to present every continuous frame $\mathcal{O}(X)$ as the rounded ideals of an information system. Note that this description is related to the development given in \cite{Kaw}; in particular we are using the notion of strength from that paper in (ii) of the definition of distributive information systems to follow. 

We use $`;'$ to denote relational composition. For example a subset $\leq \subseteq P \times P$, i.e. a relation on a set $P$, is a partial order if and only if $\leq ; \leq = \leq$, $\Delta \subseteq \leq$  and $\Delta = \leq \cap \geq$. We can similarly isolate whether a subset $I$ of $P$ is an ideal using relational composition (recall that being an ideal means lower closed and directed). 

\begin{definition}
    A \emph{distributive information system (DIS)} is a pair $(D, \prec)$ with $D$ a distributive lattice and $\prec \subseteq D \times D$ a relation interacting with the partial order $\leq$ of $D$ such that: $\prec \subseteq \leq$, $\prec;\prec = \prec$, $\prec = \prec;\leq$ and

    (i) $\Downarrow a = \{b | b \prec a \}$ is an ideal for all $a \in D$.
    
    (ii) $\forall a,b,c \in D$, $a \prec b \vee c \text{ } \Rightarrow \text{ } \exists b',c' \text{ with } a \prec b' \vee c', b' \prec b, \text{ and } c' \prec c$.
\end{definition}

Distributive information systems are the models of a geometric theory, so a category $\mathbf{DIS}$ is defined. The morphisms are $\prec$-preserving lattice homomorphisms. 

\begin{example}
    The frame of opens $\mathcal{O}X$ with the way below relation $\ll$ is a distributive information system for any locally compact locale $X$.
\end{example}

Given a distributive information system $(D,\prec)$ we can define $R\text{-}idl(D)$ the collection of \emph{rounded} ideals; that is ideals $I$ such that $a \in I$ implies $\exists b \in I$ with $ a \prec b$. 

\begin{proposition}
    $R\text{-}idl(D)$ is a continuous frame (i.e. the frame of opens of a locally compact locale) for any distributive information system $(D,\prec)$. 
\end{proposition}
\begin{proof}
Directed join is union. Bottom is $\{0 \}$. For binary join: $I \vee J = \downarrow \{ i \vee j | i \in I, j \in J \}$; to see that this is rounded, say $k \leq i \vee j$ with $i \in I $ and $j \in J$ then there exists $i' \in I$ and $j' \in J$ with $i \prec i'$ and $j \prec j'$ -- but then $k \prec i' \vee j'$ and $i' \vee j' \in \downarrow\{ i \vee j | i \in I, j \in J \}$. 

    Top is $\Downarrow 1 = \{ a | a \prec 1 \}$. For binary meets: $I \wedge J = \Downarrow \{ i \wedge j | i \in I, j \in J \}$. It is clearly rounded as $\prec$ is idempotent with respect to relational composition; to see that it is an ideal recall that $(i_1 \wedge j_1) \vee (i_2 \wedge j_2) \leq (i_1 \vee i_2) \wedge (j_1 \vee j_2)$ by distributivity of $D$. 

    Proving that binary meets distribute over directed joins is straightforward from the definition just given of binary meet as directed join is given by union. So to show that $R\text{-}idl(D)$ is a frame we need to check that it is a distributive lattice. To prove this we check that $I  \wedge (J \vee K) \subseteq (I \wedge J) \vee (I \wedge K)$. To this end say $d \prec i \wedge a$ with $ a \leq j \vee k $ with $i \in I$, $j \in J$ and $k \in K$. Then $d \prec i \wedge (j \vee k) = (i \wedge j )  \vee (i \wedge k )$ so therefore by condition (ii) in the definition of distributive information system there exists $a',b'$ such that $d \leq a' \vee b'$, $a' \prec i \wedge j $ and $b' \prec i \wedge k$. Hence $d \in (I \wedge J) \vee (I \wedge K)$.

    To see, finally, that $R\text{-}idl(D)$ is a continuous poset it is easy to see that $I = \bigcup^{\uparrow} \{ \Downarrow i | i \in I \}$ for any rounded ideal $I$ and further that $\Downarrow i \ll I$.
\end{proof}

\begin{proposition}
    $\mathcal{O}X \cong R\text{-}idl(\mathcal{O}X, \ll)$ for any locally compact locale $X$.
\end{proposition}
\begin{proof}
    Using the definition of rounded, $I = \Downarrow \bigvee^{\uparrow} I$ for any rounded ideal $I$. Further $a = \bigvee^{\uparrow} \{ b | b \ll a \}$ for any open $a$ of a locally compact locale (by definition of locally compact locale). Therefore $a \mapsto \Downarrow a$ and $I \mapsto \bigvee^{\uparrow}$ establishes the isomorphism.
\end{proof}
So, every continuous frame comes from an object of $\mathbf{DIS}$. The construction can be carried out relative to any topos $\mathcal{E}$, so that any locally compact locale $X$ over any topos $\mathcal{E}$ has as its opens the rounded ideals of some object of $\mathbf{DIS}_{\mathcal{E}}$.

\subsection{Change of base for distributive information systems}

\begin{proposition}
    If $f: \mathcal{F} \rightarrow \mathcal{E}$ is a geometric morphism between elementary toposes then for any distributive information system $(D,\prec)$ in $\mathcal{E}$, $R\text{-}idl_{\mathcal{F}}(f^*D) \cong f^{\#} R\text{-}idl_{\mathcal{E}}(D)$.
\end{proposition}
In other words if a locally compact locale $X$ in $\mathcal{E}$ has as opens the rounded ideals of $(D,\prec)$, then the pullback of $X$ to $\mathcal{F}$, i.e. $f^*X$, has as its opens the rounded ideals of $f^*(D, \prec)$. To prove the proposition we need to recall that locale pullback $f^*X$ can be calculated using $\mathcal{O}_{\mathcal{F}}(f^*X) = f^{\#}\mathcal{O}_{\mathcal{E}}(X)$ where $f^{\#}:\mathbf{dcpo}_{\mathcal{E}} \rightarrow \mathbf{dcpo}_{\mathcal{F}}$ is left adjoint to $f_*: \mathbf{dcpo}_{\mathcal{F}} \rightarrow \mathbf{dcpo}_{\mathcal{E}}$ (here $\mathbf{dcpo}$ is the category of directed complete partial orders with directed join preserving maps). See the preamble to Lemma C2.4.1 \cite{Elephant} for the construction of $f^{\#}$ relative to suplattices which generalises to dcpos (see the first few sections of \cite{towdcpo}). 
\begin{proof}
    Relational composition with $\prec$ determines an idempotent dcpo homomorphism $\prec ; (\_) : idl(D) \rightarrow idl(D)$ on the ideals of $D$. Idempotents split in $\mathbf{dcpo}$ and splittings are preserved by all functors, in particular $f^{\#}: \mathbf{dcpo}_{\mathcal{E}} \rightarrow \mathbf{dcpo}_{\mathcal{F}}$. Rounded ideals are exactly the elements of the splitting of $\prec ; (\_)$, so to complete all we need to do is to prove that $f^{\#}( \prec ; (\_))$ is equal to $f^*(\prec) ; (\_)$ up to isomorphism. This can be seen using the account of change of base given in Section 6 of \cite{towdcpo}. (This property of $f^{\#}$ is also effectively an observation needed in Joyal and Tierney's original account of localic change of base via suplattices (Ch. VI, Prop. 1 of \cite{JoyT}) and, further, is explained in the paragraph before Lemma C2.4.1 in \cite{Elephant}.)   
\end{proof}

\section{A classifying localic category for locally compact locales}
We can now state and prove our main result.
\begin{theorem}\label{main}
    The pseudo-functor $\mathfrak{LK}:\mathbf{Loc}^{op} \rightarrow \mathfrak{CAT}$ defined by 
    \begin{eqnarray*}
        X \mapsto \mathbf{LK}_{Sh(X)}
    \end{eqnarray*}
    is a lax-geometric stack. 
\end{theorem}
\begin{proof}
    Section \ref{lc_is_stack} shows that we have a stack. We therefore just need to check (i) and (ii) of Proposition \ref{lax-conditions}. The technique is now quite well known (e.g. we are really just repeating the proof of the main result of \cite{HenryTowClass}).

For (i) let $l:Sh(G_0) \rightarrow \mathbf{Set}[ \mathbb{DIS} ]$ be a localic cover of the classifying topos for distributive information systems. We can assume that $l$ is an open surjection (e.g. Theorem C5.2.1 of \cite{Elephant}). Let $C_{LK}$ be the locally compact locale corresponding to the rounded ideals of $l^*(G_{DIS}, \prec_{DIS})$ where $(G_{DIS},\prec_{DIS})$ is the generic distributive information system. So, $\mathcal{O}_{Sh(G_0)}C_{LK} \equiv R\text{-}Idl_{Sh(G_0)}(l^*(G_{DIS},\prec_{DIS}))$. Then for any locally compact locale $A$ in $Sh(X)$, $(\mathcal{O}_{Sh(X)} A, \ll_{\mathcal{O}_X A})$ is a distributive information system in $Sh(X)$ and so is classified by a geometric morphism $p_A:Sh(X) \rightarrow \mathbf{Set}[ \mathbb{DIS} ]$. The pullback (in the category of bounded toposes) of $l$ along $p_A$ is localic and an open surjection so must be of the form $q:Sh(Y_A) \rightarrow Sh(X)$ where the locale map $q$ is an open surjection (and so of effective descent in $\mathbf{Loc}$; Prop. C5.1.4 of \cite{Elephant} -- recall open surjections are pullback stable, Th. C3.1.27 of \cite{Elephant}). The following pullback diagram of geometric morphisms
\[
\begin{tikzcd}
Sh(Y_A) \arrow[r, "x"] \arrow[d, "q"'] & Sh(G_0) \arrow[d, "l"] \\
Sh(X) \arrow[r, "p_A"'] & \mathbf{Set}[\mathbb{DIS}]
\end{tikzcd}
\]
clarifies which map is $q$. We can then use the following to complete our verification of condition (i) of Proposition \ref{lax-conditions}: 
\begin{eqnarray*}
\mathcal{O}_{Sh(Y_A)}q^*A & = & q^{\#} \mathcal{O}_{Sh(X)}(A) \\
& \cong & q^{\#} R\text{-}idl_{Sh(X)} (\mathcal{O}_X(A), \ll_{\mathcal{O}_X A}) \\
& \cong & q^{\#} R\text{-}idl_{Sh(X)}  p_A^*(G_{DIS}, \prec_{DIS}) \\
& \cong & R\text{-}idl_{Sh(Y_A)} q^* p_A^*(G_{DIS}, \prec_{DIS}) \\
& \cong & R\text{-}idl_{Sh(Y_A)} x^* l^*(G_{DIS}, \prec_{DIS}) \\
& \cong & x^{\#} R\text{-}idl_{Sh(G_0)}  l^*(G_{DIS}, \prec_{DIS}) \\
& = & x^{\#} \mathcal{O}_{Sh(G_0)}  C_{LK} \\
& = & \mathcal{O}_{Sh(Y_A)} x^*C_{LK} \text{.}\\
\end{eqnarray*}

For (ii) observe that as the locale $C_{LK}$ is locally compact so is $\pi_i^*C_{LK}$ where $\pi_i: G_0 \times G_0 \rightarrow G_0$ for $i = 1,2$. But locally compact locales are exponentiable. Define $C_1$ to be the domain of $(\pi_2^*C_{LK})^{\pi_1^*C_{LK}}$ where the exponentiation is in the category of locales over $G_0 \times G_0$; i.e., equivalently, in $\mathbf{Loc}/G_0 \times G_0$.
\end{proof}

\section{Concluding remarks}

The `big idea' here is that the existence  of a classifying category for locally compact objects is, from the perspective of axiomatic locale theory (e.g. \cite{TowIdemp}), the correct notion of infinity. The train of thought is as follows: the existence of a classifying groupoid for discrete locales is equivalent to the Axiom of Infinity (Blass' Theorem), but from the perspective of axiomatic locale theory we want something that is symmetric under discrete/compact Hausdorff duality \footnote{i.e., open/proper duality, see \cite{towpara}.}. So why not take the existence of $\mathbb{C}_{\mathfrak{LK}}$ as our `localic axiom of infinity'?

This is tempting, but for the observation to have depth we must establish some properties of $\mathbb{C}_{\mathfrak{LK}}$. In the discrete case we are able to make new geometric theories from old by exploiting the existence of 2-limits in the 2-category of bounded toposes over some base topos $\mathcal{S}$. For example, establishing the Axiom of Infinity from an assumption of the existence of a classifying localic groupoid, requires constructions on stacks of geometric theories. (In detail, from the object classifier one needs to construct a classifying topos for successor algebras; see the account of Blass' Theorem given in B4.2.11 \cite{Elephant}.) Now, in this `discrete' topos case all the classifying localic groupoids can be assumed to be open and so have well-behaved connected component adjunctions; i.e. the `trivial $\mathbb{G}$-object' functor $\mathbb{G}^*:\mathbf{Loc}_{\mathcal{S}} \rightarrow [ \mathbb{G},  \mathbf{Loc}_{\mathcal{S}}]$ has a left adjoint such that the corresponding adjunction is stably Frobenius. The point of making this technical observation is that it is this property that is the key to carrying out 2-categorical constructions without relying on the finitary 2-completeness of toposes bounded over $\mathcal{S}$. In fact, quite generally, for any cartesian category $\mathcal{C}$, the category of geometric stacks on $\mathcal{C}$ is finitely 2-complete once we restrict to internal groupoids $\mathbb{G}$ that have well-behaved connected component adjunctions. (For background on this see \cite{Pronk96} and Corollary 7.4 of \cite{Noohi}; the paper \cite{TowHS} essentially covers the general statement just made.) Unfortunately I have not been able to establish whether $c\mathbb{C}_{\mathfrak{LK}}$ has a connected components adjunction. Putting this line of thought another way: is $c\mathbb{C}_{\mathfrak{LK}}$, in some sense, `bounded' over $\mathbf{Loc}$ so that we can mimic the topos theoretic approach of constructing geometric theories from a base theory of objects? I expect that $\mathbb{C}_{\mathfrak{LK}}$ is well behaved, but this needs more careful technical examination. For now we must just make do with our main result: we can classify locally compact locales.

\end{document}